\documentclass[a4paper]{article}

\usepackage{amsmath,amsfonts,amssymb,amsthm}

\usepackage[active]{srcltx}

\usepackage{cite}

\newtheorem{thm}{Theorem}[section]
\newtheorem{prop}[thm]{Proposition}

\theoremstyle{definition}

\newtheorem*{rem}{Remark}

\newcommand{\A}{\mathcal{A}}        

\newcommand{\C}{\mathbb{C}}         

\newcommand{\co}[2]{#1_{(#2)}}      
\newcommand{\cop}{\Delta}           
\newcommand{\dn}{{\mathord{\downarrow}}} 
\newcommand{\eps}{\varepsilon}      
\renewcommand{\H}{\mathcal{H}}      
\newcommand{\half}{{\mathchoice{\oh}{\oh}{\shalf}{\shalf}}} 
\DeclareMathOperator{\id}{id}       
\newcommand{\ket}[1]{|#1\rangle}    
\newcommand{\bra}[1]{\langle#1|}    
\newcommand{\lt}{\triangleright}    
\newcommand{\N}{\mathbb{N}}         
\newcommand{\oh}{{\tfrac{1}{2}}}    
\newcommand{\ooh}{{\tfrac{3}{2}}}   
\newcommand{\ox}{\otimes}           
\newcommand{\R}{\mathbb{R}}         
\newcommand{\rt}{\triangleleft}     
\newcommand{\set}[1]{\{\,#1\,\}}     
\newcommand{\Sf}{\mathbb{S}}        
\newcommand{\sg}{\sigma}            
\newcommand{\shalf}{{\scriptstyle\frac{1}{2}}} 
\newcommand{\ssesq}{{\scriptstyle\frac{3}{2}}} 
\newcommand{\sesq}{{\mathchoice{\ooh}{\ooh}{\ssesq}{\ssesq}}} 
\newcommand{\U}{\mathcal{U}}        
\newcommand{\up}{{\mathord{\uparrow}}} 
\newcommand{\Z}{\mathbb{Z}}         
\def\<#1,#2>{\langle#1,#2\rangle}   

\newcommand{\SU}{\A(\mathrm{SU}_q(2))}  
\newcommand{\podl}{\A(\mathrm{S}^2_{qc})}  
\newcommand{\pinf}{\A(\mathrm{S}^2_{q\infty})}
\newcommand{\su}{\U_q(\mathrm{su}(2))}  
\newcommand{\lin}{{\rm span}}  
\newcommand{\im}{\mathrm{i}}

\newcommand{\tT}{{\tilde T}}
\newcommand{\tW}{{\tilde W}}
\newcommand{\tH}{{\tilde\H}}
\newcommand{\tpi}{{\tilde \pi}}

\newcommand{\tJ}{{\tilde J}}
\newcommand{\tD}{{\tilde D}}

 \newcommand{\hs}{\hspace{1pt}}
 \newcommand{\hsp}{\hspace{-1pt}}

\title{Restricting the bi-equivariant spectral triple\\ on quantum SU(2) 
to the Podle\'s spheres}

\author{
{\sc Elmar Wagner
} \\
\normalsize
Instituto de F\'isica y Matem\'aticas\\
\normalsize
Universidad Michoacana de San Nicol\'as de Hidalgo, Morelia, M\'exico\\
\normalsize
e-mail: {\it elmar@ifm.umich.mx}}

\date{}      

\usepackage{geometry}
\begin{document}
\maketitle

\begin{abstract}
It is shown that the isospectral bi-equivariant spectral triple on quantum SU(2)
and the isospectral equivariant spectral triples on the Podle\'s spheres are related
by restriction. In this approach, the equatorial Podle\'s sphere is distinguished
because only in this case the restricted spectral triple admits an equivariant
grading operator together with a real structure (up to
infinitesimals of arbitrary high order). The real structure is expressed by the
Tomita operator on quantum SU(2) and it is shown that the failure of the real
structure to satisfy the commutant property is related to the failure of the
universal R-matrix operator to be unitary.  
\end{abstract}

\section{Introduction}                                 \label{sec:intro}

The search for spectral triples on noncommutative spaces arising in quantum group theory is an 
active research topic. A typical strategy for finding (equivariant) spectral triples on $q$-deformed spaces 
is a case by case study starting with a quantum analogue of the classical spinor bundle and defining 
the Dirac operator on $q$-analogues of harmonic spinors 
(see, e.g., \cite{DDLSphere,DDLCPn,DDLW,DabrowskiLPS,DabrowskiLSSV,DS}). 
Until now, only few general methods for the construction of spectral triples were found. 
The most notable examples are the construction of Dirac operators on quantum flag manifolds by 
Kr\"ahmer \cite{K} and the construction of equivariant spectral triples on compact quantum groups by 
Neshveyev and Tuset \cite{NT}. Therefore the question arises whether the latter construction 
on compact quantum groups can be used 
to find spectral triples on the associated quantum homogeneous spaces.

We approach this question by studying 
the relation between the 
bi-equivariant Dirac operator on quantum SU(2) \cite{DabrowskiLSSV} 
and spectral triples on the 1-parameter 
family of Podle\'s spheres  $\podl$, $c\in [0,\infty]$ \cite{DDLW}. 
This example exhibits already some interesting features. 
Whereas the standard Podle\'s sphere $\A(\mathrm{S}^2_{q0})$ is 
distinguished for being obtained by a quotient of quantum groups  
and admitting a rich non-commutative spin geometry \cite{W}, 
it is the equatorial Podle\'s sphere $\A(\mathrm{S}^2_{q\infty})$ 
on the other extreme which distinguishes in the present approach. 
The restriction of the bi-equivariant Dirac operator on quantum SU(2)  
to the  Podle\'s spheres $\podl$ does yield a spectral triple  
for all $c\in [0,\infty]$, but only in the case $c=\infty$ the obtained spectral triple admits 
an equivariant grading operator. 

Having an equivariant even spectral triple 
on $\A(\mathrm{S}^2_{q\infty})$, one can ask for an equivariant real structure. 
Again, our aim is to relate the real structure on $\A(\mathrm{S}^2_{q\infty})$ with the one coming from the 
spectral triple on quantum SU(2). Moreover, and maybe more interesting, we want to implement 
the real structure by the Tomita operator on $\SU$. 
It is known that an equivariant real structure for the bi-equivariant spectral triple on quantum SU(2) 
cannot satisfy the commutant and first order property exactly 
but does so up to compacts of arbitrary high order \cite{DabrowskiLSSV}. 
Starting from the Tomita operator on $\SU$, we will construct an equivariant operator on the 
quantum spinor bundle of $\A(\mathrm{S}^2_{q\infty})$ which satisfies the commutant property. 
This operator is not unitary but its unitary part coincides with restriction of the 
equivariant real structure on quantum SU(2). The construction uses the R-matrix operator of 
$\U_q(\mathrm{sl}(2))$ for intertwining tensor product representations. 
It is argued that the failure of this intertwining operator to be unitary is responsible for  
the failure of real structure to satisfy the commutant property. 

\section{Preliminaries}   
\subsection{Algebraic Preliminaries}                   \label{sec:alg-defns}

Throughout this paper, $q$ stands for real number such that $0<q<1$, and we set 
$[x] = [x]_q := \frac{q^x - q^{-x}}{q - q^{-1}}$ for $x\in \R$.
All algebras appearing in this paper will be complex and unital.
We shall use Sweedlers notation for the coproduct, namely,
$\cop x =: \co{x}{1} \ox \co{x}{2}$.

The Hopf $*$-algebra $\su$ is generated by
$e$, $f$, $k$, $k^{-1}$
with defining relations
\begin{equation*}                               
kk^{-1}=k^{-1}k=1,\quad
ek = qke,  \quad  kf = qfk, \quad
fe - ef = (q - q^{-1})^{-1}(k^2 - k^{-2}),
\end{equation*}
coproduct
$\cop k = k \ox k$,\ \,$\cop e = e \ox k + k^{-1} \ox e$,\ \,$\cop f = f \ox k + k^{-1} \ox f$,
counit $\epsilon(k)=1$, $\epsilon(f)=\epsilon(e)=0$, antipode
$S(k) = k^{-1}$, $S(f)= -q f$, $S(e)= -q^{-1} e$, and  involution $k^* = k$ and $f^* = e$. 

The coordinate Hopf $*$-algebra $\SU$ of the quantum $\mathrm{SU}(2)$
group has two generators $a$ and $b$ satisfying the relations
\begin{gather*}
ba = q ab,  \quad  b^*a = qab^*, \quad bb^* = b^*b, \quad
a^*a + q^2 b^*b = 1,  \qquad  aa^* + bb^* = 1.    
\end{gather*}
The the counit $\eps$, coproduct $\cop$
and the antipode $S$ are determined by
\begin{align*}
&\cop a = a \ox a - q\,b \ox b^*, \ \ 
\cop b = b \ox a^* + a \ox b,\quad \eps(a) = 1,\ \ \eps(b) = 0, \\
&S(a) = a^*, \ \   S(b) = - qb, \ \ 
S(b^*) = - q^{-1}b^*,\ \ S(a^*) = a.
\end{align*}
There is a dual pairing between the Hopf $*$-algebras $\,\su$ and
$\SU$ given on generators by
$$
\<k^{\pm 1}, a> = q^{\pm\half}, \quad
\<k^{\pm 1}, a^*> = q^{\mp\half}, \quad
\<f, b> = \<e, -qb^*> = 1,
$$
and zero otherwise. The left action defined by 
$h \lt x := \co{x}{1}\<h, \co{x}{2}>$ for $h\in \su$ and $x\in \SU$
satisfies
\begin{equation}                                      \label{eq:mod-alg}
h \lt (xy) =(\co{h}{1}\lt x)(\co{h}{2}\lt y),\quad h \lt 1=\epsilon(h),
\quad (h \lt x)^* = S(h)^* \lt x^*,
\end{equation}
i.e., $\SU$ is a left $\su$-module $*$-algebra. Similarly,
$x \rt h:= \<h, \co{x}{1}> \,\co{x}{2}$ 
defines a right $\su$-action on $\SU$ such that
\begin{equation}                                      \label{rmodul}
 (xy)\rt h =(x\rt \co{h}{1})(y\rt \co{h}{2}),\quad 1\rt h =\epsilon(h),
\quad (x\rt h)^* =x^*\rt S(h)^*.
\end{equation}

We follow \cite{Podles} and
define the Podle\'s quantum sphere $\podl$, $c\in [0,\infty]$,
as the $*$-algebra generated by $A=A^*$ and $B$ with relations
\begin{align*}
&BA=q^2A, \quad B^*B=A-A^2+c,\quad BB^*=q^2A-q^4A^2+c\quad
\mbox{for}~c<\infty,\\
&BA=q^2A, \quad B^*B=-A^2+1,\quad BB^*=-q^4A^2+1\quad \mbox{for}~c=\infty.
\end{align*}
The Podle\'s quantum sphere $\podl$ can be viewed as a $*$-subalgebra
of $\SU$ by setting
\begin{align*}
&B=c^{1/2}a^{*2}+a^*b-qc^{1/2}b^{2}, \quad
A=c^{1/2}b^*a^*+bb^*+c^{1/2}ab \quad \mbox{for}~c<\infty,\\
&B=a^{*2}-qb^{2}, \quad
A=b^*a^*+ab \quad \mbox{for}~c=\infty.
\end{align*}
Then the left $\su$-action on $\SU$ turns $\podl$ into a 
left $\su$-module $*$-algebra such that the elements
\begin{align*}
&x_{-1}:=q^{-1}(1+q^2)^{1/2}B,\quad
x_0:= 1-(1+q^2)A,\quad
x_{1}:=-(1+q^2)^{1/2}B^*\quad \mbox{for}~c<\infty,\\
&x_{-1}:=q^{-1}(1+q^2)^{1/2}B,\quad
x_0:= -(1+q^2)A,\quad
x_{1}:=-(1+q^2)^{1/2}B^*\quad \mbox{for}~c=\infty
\end{align*}
transform by a spin 1 representation (see Equation \eqref{eq:uqsu2-repns}).

\subsection{Equivariant representations}
                                                   \label{sec:eqvt-repns}

Let $\H$ be a Hilbert space with inner product
$\langle \cdot,\cdot\rangle$,
$V$ a dense linear subspace, and $\A$ a $*$-al\-ge\-bra.
By a $*$-representation of $\A$ on $V$,
we mean a homomorphism $\pi: \A\rightarrow\rm{End}(V)$
such that $\langle \pi(a)v,w\rangle=\langle v,\pi(a^*)w\rangle$ for all
$v,w\in V$ and $a\in\A$.


Now assume that $\A$ is a left $\U$-module
$*$-algebra, i.e., there is a left action $\lt$ of a Hopf
$*$-algebra $\U$ on $\A$ satisfying \eqref{eq:mod-alg}.
A $*$-representation $\pi$ of $\A$ on $V$ is called (left)
$\U$-equivariant if there exists a $*$-re\-pre\-sen\-ta\-tion
$\lambda$ of $\U$
on $V$ such that
\begin{equation*}                              
\lambda(h)\,\pi(x)\xi = \pi(\co{h}{1} \lt x)\,\lambda(\co{h}{2}) \xi
\end{equation*}
for all $h \in \U$, $x \in \A$ and $\xi \in V$.
We call an operator defined on $V$ equivariant if
it commutes on $V$ with $\lambda(h)$ for all $h\in\U$.  An antilinear
operator $T$ is called equivariant  if its domain of definition contains
$V$ and if it satisfies on $V$  the relation
$T\lambda(h)=\lambda(S(h)^*)T$ for all $h\in\U$.  We say that an antiunitary
operator is equivariant if it leaves $V$ invariant and if
it is the antiunitary part of the polar
decomposition of an equivariant antilinear (closed) operator.

Given $\U$ and $\A$ as above,
the left crossed product $*$-algebra $\A\rtimes \U$ is defined as the
$*$-algebra generated by the two $*$-subalgebras $\A$ and $\U$ with
cross commutation relations
\begin{equation*}
    hx=(\co{h}{1} \lt x)\co{h}{2},\quad h \in \U,\ x \in \A.
\end{equation*}
Thus $\U$-equivariant representations of $\A$ correspond to
$*$-representations of $\A\rtimes \U$. As
Hilbert space representations of $\SU\rtimes\su$ and $\podl\rtimes\su$
have been studied extensively in \cite{SchmuedgenWCPA} and
\cite{SchmuedgenWCPAPod}, we shall mainly consider equivariant
representations from this point of view.

Above definitions have their right handed counter parts.
For instance, a $*$-re\-pre\-sen\-ta\-tion $\pi$ of a right $\U$-module
$*$-algebra $\A$ (i.e.~\eqref{rmodul} is satisfied)
 is called (right)
$\U$-equivariant if there exists a $*$-re\-pre\-sen\-ta\-tion
$\rho$ of $\U$
on $V$ such that
\begin{equation}                             \label{r-covar-repn}
\pi(x)\hs \rho(h)\hs\xi =\rho(\co{h}{1})\hs \pi(x\rt \co{h}{2}) )\hs \xi,
\quad h \in \U,\,\ x \in \A,\, \ \xi \in V.
\end{equation}
Assume that we are given a left and right $\U$-equivariant representation
$\pi$ of $\A$ on $V$ such that  $\lambda(h)\rho(g)=\rho(g)\lambda(h)$
for all $h,g\in\U$.
Then we  say that an operator $X$ on $V$ is bi-equivariant if it
commutes with all operators $\lambda(h)$ and $\rho(h)$, $h\in\U$.

The irreducible $*$-representations
of $\su$ are labeled by non-negative half-in\-te\-gers.
For $l\in \half\N_0$, the corresponding representation $\sg_l$
acts on a $2l+1$-dimensional
Hilbert space $V_l$ with orthonormal basis
$\set{\ket{lm} : m = -l, -l+1,\dots, l}$ by the formulas
\begin{align} 
\begin{split}                                 
\sg_l(k)\,\ket{lm} &= q^m \,\ket{lm},\quad 
\sg_l(f)\,\ket{lm} = \sqrt{[l-m][l+m+1]} \,\ket{l,m+1}, \label{eq:uqsu2-repns} \\ 
\sg_l(e)\,\ket{lm} &= \sqrt{[l-m+1][l+m]} \,\ket{l,m-1}.
\end{split}
\end{align}
A $*$-representation of $\SU\rtimes\su$ or $\podl\rtimes\su$ 
is called integrable if its
restriction to $\su$ is a direct sum of spin $l$ representations $\sg_l$.

Suppose that $\pi$ is a $*$-representation of $\SU\rtimes\su$
(or $\podl\rtimes\su$) on $V$. Then the tensor product
representation $\pi\otimes\sg_l$ on $V\otimes V_l$ is defined by
setting $\pi\otimes\sg_l(h):=\pi(\co{h}{1})\otimes\sg_l(\co{h}{2})$
for $h\in\su$ and $\pi\otimes\sg_l(x):=\pi(x)\otimes\sg_l(1)$
for $x\in\SU$ (or  $x\in\podl$). Straightforward computations
show that $\pi\otimes\sg_l$ yields indeed a $*$-representation of
the quoted crossed product $*$-algebras.

\subsubsection{Integrable representations of $\SU\rtimes\su$}
                                            \label{sec:eqvt-repns-SU(2)}

Let $\psi$ denote the Haar state of $\SU$.
From the GNS representation of $\SU$ associated to $\psi$, we
derive a unique integrable $*$-representation $\pi_\psi$ of $\SU\rtimes\su$,
called the Heisenberg representation \cite{SchmuedgenWCPA}. It is obtained as follows.
Since $\psi$ is faithful,
we can equip $\SU$ with the inner product $\langle x,y\rangle:=\psi(y^*x)$.
The representation is given by the 
formulas $\pi_\psi(h)x=h\lt x$ and $\pi_\psi(y)x=yx$,
where $x,y\in\SU$ and $h\in\su$. 
Recall that $\SU$ has a vector-space basis
$\set{t^l_{mn} : 2l \in \N,\ m,n = -l,-l+1,\dots,l}$ 
consisting of matrix elements of its finite dimensional irreducible 
corepresentations \cite{KlimykS}. 
The  normalized matrix elements
\begin{equation}                                    \label{eq:matelt-onb}
\ket{lmn} := q^n \,[2l+1]^\half \,t^l_{nm},\quad
l\in\half\N_0,\ \, m,n=-l,-l+1,\dots,l,
\end{equation}
form an orthonormal basis for $\SU$.
On
\begin{equation}                                 \label{eq:V_ln}
V_{ln}:=\lin\set{\ket{lmn} : m=-l,-l+1,\dots,l },
\end{equation}
the restriction of $\pi_\psi$ to $\su$ becomes a spin $l$ representation,
so $\pi_\psi$ is integrable.

It follows from \cite[Proposition 1.2]{SchmuedgenWCPAPod} that
each integrable $*$-representation of the crossed product $*$-algebra
$\SU\rtimes\su$ is unitarily
equivalent to a direct sum of Heisenberg representations. Moreover, an
integrable $*$-representation of $\SU\rtimes\su$ is irreducible if and
only if the  vector space of invariant vectors (i.e., vectors belonging to a
spin 0 representations) is 1-dimensional. In particular, each irreducible
integrable $*$-representation of $\SU\rtimes\su$ is unitarily equivalent
to the Heisenberg representation.

Defining
\begin{equation}                                     \label{rho}
\rho_\psi(h)x=x\rt S^{-1}(h),\qquad x\in\SU,
\end{equation}
the left $\su$-equivariant representation $\pi_\psi$ can also be viewed
as right $\su$-equiva\-riant.
One easily shows (see, e.g.,~\cite{SchmuedgenWCPA}) that
$V_{lm}=\lin\{\ket{lmn}:n=-l,\ldots\hsp,l\}$ is an
irreducible spin $l$ representation space with highest weight
$\ket{lml}$, i.e., $\rho_\psi(k)\ket{lmn}=q^{-n}\ket{lmn}$.
Since left and right $\su$-action on $\SU$ commute, we have 
obviously 
$\pi_\psi(h)\rho_\psi(g)=\rho_\psi(g)\pi_\psi(h)$ for all
$g,h\in\su$.

\subsubsection{Integrable representations of $\podl\rtimes\su$}
                                           \label{sec:eqvt-repns-podl}

The integrable representations of 
$\podl\rtimes\su$ were completely
classified in \cite{SchmuedgenWCPAPod}. It turned out that each
integrable representation 
is a direct sum of
irreducible ones. The inequivalent irreducible integrable representation
$\pi_j$ of $\podl\rtimes\su$ are labeled by half-integers
$j\in\half\Z$. Each representation $\pi_j$ can be realized
on an invariant subspace $M_j\subset\SU$ by restricting the
Heisenberg representation $\pi_\psi$ of $\SU\rtimes\su$ to the
$*$-subalgebra $\podl\rtimes\su$. Moreover, $\SU$ is the orthogonal
direct sum of these  invariant subspaces, i.e.,
$\SU=\oplus_{j\in\half\Z}M_j$. As a left $\podl$-module, $M_j$
is finitely generated and projective. It is known that $M_j$
can be considered as a line bundle over the quantum sphere
$\mathrm{S}^2_{qc}$ with winding number $2j$
\cite{BM,H,MS}.

For the convenience of the reader, we recall from
\cite{SchmuedgenWCPAPod} the explicit description of the
irreducible
representations $\pi_j$, $j\in\half\Z$.
The Hilbert space
is the orthogonal direct sum $\bigoplus_{l=|j|,|j|+1,\ldots} V^l$,
where $V^l$ is a spin $l$-representation space with an orthonormal
basis of weight vectors  $\{ v^{l}_{k,j}:k=-l, -l+1, \ldots ,l\}$.
The generators $e$, $f$, $k$ of $\su$ act on
$V^{l}$ by (\ref{eq:uqsu2-repns}).
The actions of the generators $x_1$, $x_0$, $x_{-1}$ of
$\podl$ are
determined by
\begin{align}
\pi_j(x_1) v^l_{k,j} =\  &q^{-l+k} [l\!+\!k\!+\!1]^{1/2}
[l\!+\!k\!+\!2]^{1/2}
[2l\!+\!1]^{-1/2} [2l\!+\!2]^{-1/2} \alpha_j(l) v^{l+1}_{k+1,j}
\nonumber\\
        &\qquad -q^{k+2} [l\!-\!k]^{1/2} [l\!+\!k\!+\!1]^{1/2} [2]^{1/2} [2l]^{-1}
\beta_j (l) v^l_{k+1,j}                                      \label{x1}\\
           &\qquad\qquad -q^{l+k+1} [l\!-\!k\!-\!1]^{1/2} [l\!-\!k]^{1/2}
[2l\!-\!1]^{-1/2} [2l]^{-1/2} \alpha_j (l\!-\!1) v^{l-1}_{k+1,j},\nonumber
\end{align}
\begin{align}
\pi_j(x_0) v^l_{k,j} =\ &q^k [l\!-\!k\!+\!1]^{1/2} [l\!+\!k\!+\!1]^{1/2}
[2]^{1/2}
[2l\!+\!1]^{-1/2} [2l\!+\!2]^{-1/2} \alpha_j(l) v^{l+1}_{k,j}\nonumber\\
           &\qquad +\big(1-q^{l+k+1} [l\!-\!k][2][2l]^{-1}\big)
            \beta_j (l) v^l_{k,j}                            \label{x0}\\
           &\qquad\qquad +q^k [l\!-\!k]^{1/2} [l\!+\!k]^{1/2} [2]^{1/2}
[2l\!-\!1]^{-1/2} [2l]^{-1/2} \alpha_j (l\!-\!1) v^{l-1}_{k,j} \nonumber       
\end{align}
%
and $\pi_j(x_{-1})=-q^{-1}\pi_j(x_{-1})^*$. For $c<\infty$,
the real numbers $\beta_j(l)$ and $\alpha_j(l)$
are defined by
\begin{align*}                                           
\beta_j(l)
&= [2l\!+\!2]^{-1}\big([2|j|](q^{-2} \lambda_\pm - \lambda_\mp)
+(1-q^{-2})[|j|]\hspace{1pt}[|j|\!+\!1]
- (1-q^{-2})[l][l\!+\!1]\big), \\
\alpha_j(l)&=[2]^{-1/2}[2l\!+\!3]^{-1/2}[2l\!+\!2]^{1/2}
\big(1+[2]^2c-(1-q^{2}) \beta_j (l)
-q^2(\beta_j (l) )^2\big)^{1/2},            
\end{align*}
where $\lambda_\pm = 1/2 \pm (c+ 1/4)^{1/2}$. 
For $c=\infty$, $\beta_j(l)$ and $\alpha_j(l)$
are given by
\begin{equation*}                                      
\beta_j(l)
\hsp=\hsp \mathrm{sign}(j)\hs q^{-1} [2l\hsp+\hsp 2]^{-1}[2]\hs [2|j|], \ 
\alpha_j(l)\hsp=\hsp[2]^{-1/2}[2l\hsp+\hsp3]^{-1/2}[2l\hsp+\hsp2]^{1/2}
\big([2]^2-q^2(\beta_j (l) )^2\big)^{\hsp 1/2}\hsp.    
\end{equation*}
In the case $l=k=j=0$, Equation \eqref{x0} becomes 
$\pi_0(x_0) v^0_{0,0} =  \alpha_0(0)\hs v^{1}_{0,0}$. 

In the present paper, we are particularly interested in the
representation $\pi_0$ acting on the trivial line bundle
$M_0\cong\podl$. This representation can also be obtained from
the GNS representation associated to Haar state $\tilde\psi$
on $\podl$. By the uniqueness of the Haar state, one can take $\tilde\psi$ 
to be the restriction of $\psi$ on $\SU$ to $\podl$. 
Analogously to the Heisenberg representation
from the previous subsection, we have
$\langle x,y\rangle:=\tilde \psi(y^*x)$, $\pi_0(y)x=yx$, 
and  $\pi_0(h)x=h\lt x$,
where $x,y\in\podl$ and $h\in\su$.

\subsection{Spectral triples}
                                                   \label{sec:spec-tr}

By a (compact) spectral triple $(\A,\H,D)$, we mean a $*$-algebra $\A$,
a bounded $*$-representa\-tion $\pi$ of $\A$ on a Hilbert space $\H$, and a
self-adjoint operator $D$ on $\H$ such that \cite{CBook}
\begin{enumerate}
\item[$(i)$]
     $(D-\zeta)^{-1}$ is a compact operator for all $\zeta\in
\C\setminus\R$,
\item[$(ii)$]
     the commutators $[D,\pi(a)]$ are bounded for all $a\in\A$.
\end{enumerate}
If there exists an $n\in\N_0$ such that
the asymptotic behavior of the eigenvalues
$0\leq\mu_1\leq\mu_2\leq\dots$ of $|D|^{-n}$
is given by $\mu_k={\rm O}(k^{-1})$ as $k\rightarrow\infty$, then
the spectral triple is said to be $n^+$-summable.

Let $\U$ be a Hopf $*$-algebra, $\A$ a left and/or right
$\U$-module and
$\pi$ a $\U$-equivariant representation on $\H$.
The spectral triple $(\A,\H,D)$ is called  left or right $\U$-equivariant
if $D$ is a left or right equivariant operator.
We call it bi-equivariant if $D$ is left and right $\U$-equivariant.

An (equivariant) spectral triple $(\A,\H,D)$ is called even if
there exists   an (equivariant) grading operator $\gamma$ on $\H$ such that
$\gamma^*=\gamma$,
$\gamma^2=1$, $\gamma D=-D\gamma$, and $\gamma\pi(a)=\pi(a)\gamma$
for all $a\in\A$. 

In the seminal paper \cite{ConnesReal}, a real structure $J$ on a spectral
triple was defined by an antiunitary operator $J$ on $\H$ satisfying 
\begin{equation}                                     \label{eq:aJb=0}
[\pi(x), J\pi(y)J^{-1}]=0, \quad\
[[D,\pi(x)],J\pi(y)J^{-1}]=0, \quad\  x,y\in\A,
\end{equation}
$J^2=\pm 1$, $JD=\pm DJ$ and, for even spectral triples, 
$J\gamma =\pm \gamma J$. 
The real structure is called equivariant, if $J$ is equivariant in the sense of 
Section \ref{sec:eqvt-repns}. 

It was noted in \cite{DabrowskiLSSV} that, by requiring equivariance of
$J$, it is not possible to satisfy \eqref{eq:aJb=0}. However,
the problem was overcome in \cite{DabrowskiLSSV} by requiring that
\eqref{eq:aJb=0} holds up to an operator ideal contained
in the ideal of infinitesimals of arbitrary high order.
Here, a compact operator $A$ is called an infinitesimal of arbitrary high
order if its singular values $s_n(A)$ satisfy
$\lim_{n\to\infty}n^p s_n(A)=0$ for all $p>0$.

\section{Equivariant spectral triples}
                                              \label{sec:Dirac-oper}

\subsection{The equivariant Dirac operator on $\SU$}
                                              \label{sec:Dirac-oper-SU}

In this section, we summarize the results
from \cite{DabrowskiLSSV} concerning the equivariant isospectral
Dirac operator on  $\SU$.
Starting point of the construction 
is the Heisenberg
representation $\pi_\psi$ of the left crossed product $*$-algebra
$\SU\rtimes\su$ on $V:=\SU$.
The spin representation $\pi$ is given by the tensor product
representation $\pi:=\pi_\psi\otimes\sg_\half$ acting on
$W:=V\otimes V_\half$.
The Hilbert space completion of $W$ will be denoted by $\H$.
Setting $\rho\hsp :=\hsp\rho_\psi\otimes\id$, the left
$\su$-\hspace{0pt}equivariant representation $\pi$ becomes also right
$\su$-equivariant
and we have  $\pi(h)\rho(g)=\rho(g)\pi(h)$ for all $h,g\in\su$.

Recall that the set of  vectors defined in \eqref{eq:matelt-onb} forms
an orthonormal basis for $\SU$. Set
\begin{equation}                                \label{eq:Hl}
H_l:=\lin\set{\ket{lmn} : m,n =-l,-l+1,\dots,l }.
\end{equation}
Then,
by \eqref{eq:V_ln},
$H_l=\oplus_{n=-l}^l V_{ln}$ is the $2l+1$-fold orthogonal sum of
irreducible spin l representation spaces.
As before, let $V_l$, $l\in\half\N_0$,
denote the irreducible spin $l$ representation space.
From the Clebsch-Gordan decomposition, it is
known that
\begin{equation}                         \label{eq:CGDec}
 V_l\otimes V_\half = V_{l-\half}\oplus V_{l+\half},\ \ l=\half,1,\dots,
\qquad V_0\otimes V_\half =V_\half.
\end{equation}
Hence we can write
\begin{equation}                         \label{eq:Wl}
H_l\otimes V_\half=W_{l-\half}^\up \oplus W_{l+\half}^\dn,
\ \  l=\half,1,\dots,
\qquad H_0\otimes V_\half=W_\half^\dn,
\end{equation}
where $W_{l-\half}^\up$ and $ W_{l+\half}^\dn$ are the linear
spaces of vectors from $H_l\otimes V_\half$
belonging to spin $l-\half$ and  spin $l+\half$
representations, respectively.
Since $V=\oplus_{l\in\half\N_0}H_l$, it follows that the representation space
$W=V\otimes V_\half$ decomposes into
\begin{equation}                         \label{eq:W}
  W= \bigoplus_{l\in\half\N_0} W_l^\up\oplus\bigoplus_{l\in\half\N}W_l^\dn.
\end{equation}
By \eqref{eq:Hl}--\eqref{eq:Wl}, we have
$\dim W_l^\up= (2l + 1)(2l + 2)$ and
$\dim W_{l}^\dn= 2l(2l + 1)$.

Now consider the self-adjoint operator $D$ on $\H$ determined by
\begin{equation}                                            \label{eq:D}
D w_l^\up=(l+\half)w_l^\up,\ \, w_l^\up\in W_l^\up, \qquad
D w_{l}^\dn=-(l+\half)w_{l}^\dn,\ \,
w_{l}^\dn\in  W_{l}^\dn.
\end{equation}
It was proved in \cite{DabrowskiLSSV} that
$(\SU,\H,D)$ is a bi-equivariant spectral triple.
The eigenvalues of $D$ are $l+\half$ with multiplicities
$(2l + 1)(2l + 2)$
and $-(l+\half)$ with multiplicities $2l(2l + 1)$, where
$l\in\half\N_0$ and
$l\in\half\N$, respectively.
Using the results from \cite{Hitchin}, one easily checks that
the eigenvalues and multiplicities of $2D-\half$ coincide with
those of a classical Dirac operator on $\Sf^3\approx\mathrm{SU}(2)$
equipped with a
$\mathrm{SU}(2)\times\mathrm{SU}(2)$-invariant metric
(set $\lambda=-1$ in \cite[Proposition 3.2]{Hitchin}).
So we have an isospectral deformation of a $\mathrm{SU}(2)$-bi-invariant
classical spectral triple.

\subsection{The equivariant Dirac operator on $\podl$}
                                              \label{sec:Dirac-oper-podl}

Our aim is to show that restricting the Dirac operator on $\SU$ to 
the quantum spinor bundle $\podl\otimes V_\half\subset\SU\otimes V_\half$ 
yields a spectral triple on $\podl$.

To begin, recall from Section \ref{sec:eqvt-repns-podl} that $\pi_0$ is
the $*$-representation  of $\podl\rtimes\su$
obtained by restricting the
Heisenberg representation of $\SU\rtimes\su$
to its subalgebra $\podl\rtimes\su$ and to the
subspace $M_0=\podl$ of $V=\SU$.
Along the lines of
the previous subsection, we take the tensor product representation
$\tilde\pi
:=\pi_0\otimes\sg_\half$ on $\tilde W:=M_0\otimes
V_\half$  as spin representation.
Furthermore, the Hilbert space completion of
$\tilde W$, say $\tH$, will be considered as Hilbert space
of spinors.

Let $\tilde V_l:=\lin\set{v^l_{m,0}\hsp :\hsp m=-l,\dots,l }$. 
In Section \ref{sec:eqvt-repns-podl}, we saw that
$M_0=\oplus_{l\in\N_0}\tilde V_l$ is an
orthogonal sum of irreducible spin $l$ representation spaces.
The Clebsch-Gordan decomposition yields 
\begin{equation}                         \label{eq:tWl}
\tilde V_l\otimes V_\half=\tilde W_{l-\half}^\up \oplus
\tilde W_{l+\half}^\dn,\ \  l=1,2,\dots,
\qquad\tilde V_0\otimes V_\half=\tilde W_\half^\dn,
\end{equation}
where the restriction of $\pi_0\otimes\sg_\half$ to
$\tilde W_{l}^\up$ or $ \tilde W_{l}^\dn$ is an
irreducible spin $l$ representation of $\su$. Clearly,
\begin{equation}                         \label{eq:tW}
  \tilde W= \bigoplus_{l\in\N_0} \tilde W_{l+\half}^\up\oplus
\tilde W_{l+\half}^\dn.
\end{equation}
As $M_0\!=\!\oplus_{l\in\N_0}\tilde V_l\subset
V\!=\!\oplus_{l\in\half\N_0}H_l$ and $\tilde V_l$ is
a spin $l$ representation space, we have
$\tilde V_l\subset H_l$.
Comparing  \eqref{eq:Wl} and \eqref{eq:tWl}
shows that $\tilde W_l^\up\subset W_l^\up$ and
$\tilde W_l^\dn\subset W_l^\dn$, where $l=\half,\sesq ,\dots$. 
The operator $D$ from
Subsection \ref{sec:Dirac-oper-SU}  acts on each $W_l^\up$ and
$W_l^\dn$ as a multiple of the identity. In particular, $D$
leaves the subspaces $\tilde W_l^\up$ and $\tilde W_l^\dn$ invariant.
Let $\tilde D$ denote (the closure of) the restriction of $D$ to $\tilde W$. 
By \eqref{eq:D}, 
\begin{equation}                                        \label{eq:tD}
\tilde D \tilde w_l^\up=(l+\half)\tilde w_l^\up,\ \,
\tilde w_l^\up\in \tilde W_l^\up, \qquad
\tilde D \tilde w_{l}^\dn=-(l+\half)\tilde w_{l}^\dn,\ \,
\tilde w_{l}^\dn\in  \tilde W_{l}^\dn,\qquad l=\half,\sesq ,\dots\,.
\end{equation}
Since the spin representation $\pi_0\otimes\sg_\half$
is obtained by restricting
$\pi_{\psi}\otimes\sg_\half$ to the subalgebra $\podl\rtimes\su$ of
$\SU\rtimes\su$ and to the subspace $\tilde W\subset W$, we can now
apply verbatim the results from \cite{DabrowskiLSSV}. Therefore,
$[\tilde D,\pi_0\otimes\sg_\half(x)]$ is bounded for all
$x\in\podl$ and $\tilde D$ is left $\su$-equivariant because  
the same is true for 
$[D,\pi_{\psi}\otimes\sg_\half(x)]$ and $D$.

Comparing the eigenvalues and the corresponding multiplicities with 
those of the Dirac operator on the Riemannian 2-sphere with the standard metric
(see, e.g., \cite{Polaris}), one sees that $(\podl,\tH,\tilde D)$ is
an isospectral deformation of the classical spectral triple. 
From the asymptotic behavior of the eigenvalues, one readily concludes that it is 
$2^+$-summable. 

Summarizing, we arrive at the following proposition.
\begin{prop}                                             \label{T1}
Restricting the spectral triple $(\SU,\H,D)$ 
to  $\podl\otimes V_\half\subset\SU\otimes V_\half$ 
(considered as subspaces of $\H$) 
gives rise to a left $\su$-equivariant,
$2^+$-summable spectral triple $(\podl,\tH,\tilde D)$, 
where $\tH$ denotes the closure of $\podl\otimes V_\half$ in $\H$. 
It is an isospectral deformation of the classical
spectral triple on the Riemannian 2-sphere with the standard metric.
\end{prop}

\begin{rem}
The fact that $(\podl,\tH,\tilde D)$ defines a left $\su$-equivariant spectral 
triple on the Podle\'s spheres has been proved in \cite{DabrowskiLPS} for $c=\infty$ 
and in \cite{DDLW} for all $c\in [0,\infty]$ by direct computations. 
\end{rem}


%
%
\subsection{Equivariant grading operator on $(\pinf,\tH,\tD)$}
                                              \label{sec:grading}
%
As $(\podl,\tH,\tilde D)$ is  an isospectral deformation
and $2^+$-summable, its classical dimension  is~2. For this reason and in
analogy with the classical picture,  we are interested in obtaining an
even spectral triple. 
The next proposition shows that an equivariant grading operator exists only
for the spectral triple $(\pinf,\tH,\tD)$. The ``only if'' part of the 
proposition is an interesting result since 
it seems to contradict \cite{DDLW}, where equivariant even spectral triples 
for all Podle\'s spheres were constructed. In fact, one of the main 
purposes of this paper is to point out that the construction of spectral triples 
by restriction may be possible but extra care has to be taken when trying to 
satisfy additional structures, for instance, when passing from odd to even ones. 
The reason behind the seeming contradiction between \cite{DDLW} and 
Proposition \ref{gradprop} will be explained after the proof 
of the proposition. Roughly speaking, it arises because 
the spectral triples from \cite{DDLW} and Proposition \ref{T1} are 
unitarily equivalent but, for $c<\infty$, the unitary operators 
implementing the equivalence are not 
compatible with the (unique) equivariant grading operator. 
\begin{prop}                                         \label{gradprop}
The spectral triple
$(\podl,\tH,\tilde D)$ from Proposition \ref{T1} admits  an
equivariant grading operator of an even spectral triple if and only if
$c=\infty$.
\end{prop}

\begin{proof}
The assumptions on $\gamma$ imply that it commutes with all
elements from the crossed product algebra $\podl\rtimes\su$.
In \cite[Proposition 4.4]{SchmuedgenWCPAPod}, it has been shown that the
tensor product representation $\pi_0\otimes\sg_\half$ on $M_0\otimes
V_\half$ decomposes into
the direct sum of the irreducible representations
$\pi_{-\half}$ and $\pi_\half$ on $M_{-\half}$ and $M_\half$,
respectively. From $\gamma^*=\gamma$ and $\gamma^2=1$, it follows that
$\gamma$ has eigenvalues $\pm 1$. Since the irreducible representations
$\pi_{-\half}$ and $\pi_\half$ are nonequivalent and integrable,
we conclude that $\gamma$ acts on $M_{-\half}$ and
$M_\half$  by $\pm\id$, with opposite sign on each space.
In the notation of Section \ref{sec:eqvt-repns-podl}, we can
assume without loss of generality that
$\tW=M_{-\half}\oplus M_{\half}$ and
\begin{equation}                                    \label{gamma}
\gamma\hs v^l_{m,-\half}=- v^l_{m,-\half},\ \quad
\gamma\hs v^l_{m,\half}= v^l_{m,\half}.
\end{equation}

Next, the relation $\gamma \tD=-\tD\gamma$ forces $\gamma$ to
map $\tW_l^\dn $ into $\tW_l^\up$ and $\tW_l^\up$ into $\tW_l^\dn$.
In addition, the equivariance of $\gamma$ implies that we can choose a
basis
$\{\ket{lm\dn}: m=-l,\ldots,l\}$ for $\tW_l^\up$ and a basis
$\{\ket{lm\up}: m=-l,\ldots,l\}$ for $\tW_l^\up$ such
that the action of
$\su$ on these vectors is given by \eqref{eq:uqsu2-repns} and
\begin{equation}                                    \label{gammaket}
\gamma\hs\hs \ket{lm\dn}= \ket{lm\up},\ \quad
\gamma\hs\hs \ket{lm\up} = \ket{lm\dn}.
\end{equation}

Assume now that $\gamma$ is an operator on $\tW$ satisfying
Equations \eqref{gamma} and \eqref{gammaket}.
Using $\lin\{v^\half_{\half,-\half},v^\half_{\half,\half}\}=
\lin\{\ket{\half,\half,\dn},\ket{\half,\half,\up}\}$
and applying \eqref{gamma} and \eqref{gammaket}, we can write
\begin{equation*}
\ket{\half,\half,\dn}=
s\, v^\half_{\half,\half} + t\,v^\half_{\half,-\half},\qquad
\ket{\half,\half,\up}=
s\, v^\half_{\half,\half} - t\,v^\half_{\half,-\half},
\end{equation*}
where $s,t\in \C$ such that $|s|^2+|t|^2=1$.
Moreover, $\langle \half,\half,\dn\hs|\hs\half,\half,\up \rangle=0 $
implies $|s|^2=|t|^2=\half$.
In the notation of \eqref{eq:uqsu2-repns}, let
$V_\half= \lin\{ \ket{\half,-\half},\ket{\half,\half}\}$.
From \eqref{eq:tWl}, it follows that
$\ket{\half,\half,\dn}
=\exp(\im\omega)\,v^0_{0,0}\otimes \ket{\half,\half}$,
where $\omega\in[0,2\pi)$. Applying the formulas from
Section \ref{sec:eqvt-repns-podl}, we compute
\begin{equation}                         \label{contra}
0=\langle v^0_{0,0}, \pi_0(x_0)v^0_{0,0}\rangle=
\langle \half,\half,\dn\hs|\tpi(x_0)|\hs\half,\half,\dn \rangle
=\half \big( \beta_\half (\half) + \beta_{-\half} (\half)\big).
\end{equation} 
For $c<\infty$, we obtain a contradiction since
$ \beta_\half (\half) + \beta_{-\half} (\half)=[3]^{-1}(q^{-2}-1)\neq 0$.
Therefore a grading operator satisfying Equations \eqref{gamma} and
\eqref{gammaket} can only exist in the case of
the equatorial Podle\'s sphere $\pinf$.

Let $c=\infty$. Our aim is to find orthonormal
vectors $\ket{lm\dn},
\ket{lm\up}\in\lin\{v^l_{m,-\half},v^l_{m,\half}\}$
such that $\gamma$ is given  by \eqref{gammaket}. To begin, consider
\begin{align}                                            \label{lldn}
\begin{split}
\ket{ll\dn}&:=v^{l-\half}_{l-\half,0}\otimes \ket{\half,\half}, \\
\ket{ll\up}&:= [2l\hsp+\hsp 2]^{-\half}                   
\big(\hsp-\hsp q^\half[2l\hsp+\hsp 1]^\half \,
v^{l+\half}_{l+\half,0}\otimes \ket{\half,-\half}
+q^{-l-\half}\,v^{l+\half}_{l-\half,0}\otimes \ket{\half,\half} \big).
\end{split}
\end{align}
Since $\ket{ll\dn}$ and $\ket{ll\up}$ are highest weight vectors of weight
$q^l$, it follows that both belong to $\lin\{v^l_{l,-\half},v^l_{l,\half}\}$.
Moreover, by \eqref{eq:tWl}, $\ket{ll\dn}\in\tW_l^\dn$ and
$\ket{ll\up}\in\tW_l^\up$.
We claim that, for some $\omega_l,\ \phi_l\in[0,2\pi)$, 
\begin{equation}                                    \label{claim}
\mbox{$\frac{1}{\sqrt{2}}$}(\ket{ll\dn}+\ket{ll\up})=
\exp(\im \omega_l)\hs v^l_{l,\half},\qquad
\mbox{$\frac{1}{\sqrt{2}}$}(\ket{ll\dn}-\ket{ll\up})=
\exp(\im \phi_l)\hs v^l_{l,-\half}. 
\end{equation}
Before justifying the claim, we observe that, for
$w=x\hs v^l_{l,\half}+y\hs v^l_{l,-\half}$ with $x,y\in \C$ and
$|x|^2+|y|^2=1$,  we have
$w=\exp(\im \omega_l)\hs v^l_{l,\pm\half}$ if and only if 
$\langle w,\pi(x_0)w\rangle = \beta_{\pm\half} (l)$. 
This is apparent from the equality 
$\langle w,\pi(x_0)w\rangle
=|x|^2\beta_\half (l)+|y|^2\beta_{-\half}(l)$ since
$\beta_\half (l)>0$ and $\beta_{-\half}(l)<0$. 
Applying the formulas from
Subsection \ref{sec:eqvt-repns-podl} (note that $\beta_0(l)=0$ for all
$l\in\N_0$) gives 
\begin{align*}
&\half(\bra{ll\dn}\pm\bra{ll\up}\hs)\hs\tpi(x_0)\hs
(\hs\ket{ll\dn}\pm\ket{ll\up})\\
 &{}\hspace{80pt} =\half [2l\hsp+\hsp 2]^{-\half}q^{-l-\half}\big(\hsp \pm\hsp 
\langle v^{l+\half}_{l-\half,0}, \pi_0(x_0)
v^{l-\half}_{l-\half,0}\rangle \pm
\langle v^{l-\half}_{l-\half,0}, \pi_0(x_0)
v^{l+\half}_{l-\half,0}\rangle \big)\\
&{}\hspace{80pt} =\pm q^{-1} [2l\hsp+\hsp 2]^{-1}[2]=\beta_{\pm\half} (l)
\end{align*}
from which the claim follows.

With $e$ denoting one of the generators of $\su$, set
\begin{equation}                                      \label{lmupdn}
\ket{lm\dn}:=||\tpi(e)^{l-m}\hs \ket{ll\dn}||^{-1}\hs
\tpi(e)^{l-m}\hs\ket{ll\dn},
\qquad
\ket{lm\up}:=||\tpi(e)^{l-m}\hs
\ket{ll\up}||^{-1}\hs \tpi(e)^{l-m}\hs \ket{ll\up}.
\end{equation}
Then \eqref{claim} implies that
\begin{equation}                                       \label{lmv}
\mbox{$\frac{1}{\sqrt{2}}$}(\ket{lm\dn}+\ket{lm\up})=
\exp(\im \omega_l)\hs v^l_{m,\half},\qquad
\mbox{$\frac{1}{\sqrt{2}}$}(\ket{lm\dn}-\ket{lm\up})=
\exp(\im \phi_l)\hs v^l_{m,-\half},
\end{equation}
and the operator $\gamma$ given by Equation \eqref{gammaket} satisfies
\eqref{gamma}. Clearly, this operator meets all the requirements 
on an equivariant grading operator.
\end{proof}
Let us now explain why $(\podl,\tH,\tilde D)$ does not admit an equivariant 
grading operator for $c<\infty$ although equivariant even spectral triples 
with the same spectral properties were constructed in \cite{DDLW} for all $c$. 
The Dirac operators from \cite{DDLW} are unitarily equivalent to $\tilde D$, 
and it follows from \cite[Equation (5.1)]{DDLW} that the unitary equivalence is 
determined by unitary operators 
$$
  U_l \,:\, \lin\{ \ket{ll\dn}, \ket{ll\up}\}\, \longrightarrow \, 
\lin\{ v^l_{l,-\half},v^l_{l,\half}\},\qquad l=\half,\sesq ,\dots\,.
$$
Now Equation \eqref{claim} tells us that the existence of an 
equivariant grading operator requires that the unitary transformations between 
$\lin\{ \ket{ll\dn}, \ket{ll\up}\}$  and 
$\lin\{ v^l_{l,-\half},v^l_{l,\half}\}$ are given by matrices of the type  
$$
\mbox{$\frac{1}{\sqrt{2}}$}
\begin{pmatrix}
 \exp(\im \omega_l) & \exp(\im \omega_l)\\
\exp(\im \phi_l) & -\exp(\im \phi_l)
\end{pmatrix}, \qquad \omega_l, \phi_l\in[0,2\pi),
$$
but the contradiction obtained below Equation \eqref{contra} proves that, 
for $c<\infty$, the matrix corresponding to the unitary operator $U_{\half}$ 
does not have the above form.

%
\subsection{The real structure on $(\SU,\H,D)$}
                                              \label{sec:real-struct-SU}
%
%
In this section, we give a brief summary of the results of \cite{DabrowskiLSSV} 
on the real structure. 
Set
\begin{equation}                      \label{eq:spinor-basis-coeffs}
C_{jm} := q^{-(j+m)/2}\,[j-m]^{1/2}\, [2j]^{-1/2},  \qquad
S_{jm} := q^{(j-m)/2} \,[j+m]^{-1/2}\,[2j]^{-1/2}.
\end{equation}
With $\ket{lmn}$ defined in \eqref{eq:matelt-onb}
and $\{\ket{\half,-\half},\ket{\half,\half}\}$ being a orthonormal basis
of $V_\half$, let
\begin{align}                      \label{eq:spinor-basis-dn}
\ket{j\hs m\hs  \nu\dn}
&:= C_{jm } \,\ket{j-\half, m +\half, \nu} \ox \ket{\half,-\half}
+ S_{jm } \,\ket{j-\half, m -\half, \nu} \ox \ket{\half,+\half}, \\
\ket{j\hs m\hs  \mu\up}                    \label{eq:spinor-basis-up}
&:= - S_{j+1,m } \,\ket{j+\half, m +\half, \mu} \ox \ket{\half,-\half}
+ C_{j+1,m } \,\ket{j+\half, m -\half, \mu} \ox \ket{\half,+\half}.
\end{align}
According to the Clebsch-Gordan decomposition of the tensor
product  representation
$\sigma_l\otimes\sigma_\half$ on $V_l\otimes V_\half$,
we have
\begin{align}                      \label{basis-dn}
W^\dn_j&=\lin \{\, \ket{j\hs m\hs\nu\dn}\,:\, m  = -j,\dots,j,\ \,
\nu = -j+\half,\dots,j-\half \,\}, \quad j=\half,1,\ldots,\\
W^\up_j&=\lin \{\,\ket{j\hs m\hs  \mu\up} \,:\, m  = -j,\dots,j,\ \,
\mu = -j-\half,\dots,j+\half\,\},
 \quad j=0,\half,\ldots,                          \label{basis-up}
\end{align}
and the sets on the right hand side are orthonormal bases.
The set of all vectors $\ket{j\hs m\hs\nu\dn}$ and
$\ket{j\hs m\hs \mu\up}$ forms an orthonormal basis for
the Hilbert space of spinors $\H$. Define an antiunitary operator $J$
on $\H$ by
\begin{equation}                                 \label{eq:J-formula}
J\, \ket{jm  n\up} = \mathrm{i}^{2(2j+m +n)}
\,\ket{j,-m ,-n,\up},\quad  J\, \ket{jm  n\dn}
= \mathrm{i}^{2(2j-m -n)} \,\ket{j,-m ,-n,\dn}.
\end{equation}
The proof of the following facts can be found in
\cite{DabrowskiLSSV}:

Let $(\SU,\H,D)$ be the spectral triple from
Subsection \ref{sec:Dirac-oper-SU}. The antiunitary operator $J$ defined
above satisfies $J^2=-1$, $JD=DJ$ and
\begin{equation}                                \label{Jequi}
J\pi(h)J^{-1}=\pi(kS(h)^*k^{-1}),\quad\  h\in\su.
\end{equation}
The commutators
$[\pi(x), J\pi(y)J^{-1}]$ and $ [[D,\pi(x)],J\pi(y)J^{-1}]$
are infinitesimals of arbitrary high order for all $x,y\in\SU$, 
so $J$ satisfies in this sense the condition 
of a real structure on the spectral triple $(\SU,\H,D)$. Moreover, $J$ is
equivariant since we can consider it, for instance, as the antiunitary
part of (the closure of) the equivariant antilinear operator $J\hs \pi(k)$.

The assembly 
$(\SU,\H,D,J)$ is viewed as an equivariant real spectral triple
on $\SU$.

%
\subsection{Implementation of the real structure by the Tomita operator on $\SU$}
                                           \label{sec:real-struct-Tomita}
%

For a GNS-re\-pre\-sen\-ta\-tion of a von Neumann algebra,
the modular conjugation from the
Tomita-Takesaki theory \cite{Takesaki}
can be used to introduce a reality operator \cite{ConnesReal}.
The objective of this section is to relate the real structure $J$
on the Hilbert space of spinors to the
modular conjugation associated with the  GNS-re\-pre\-sen\-ta\-tion
$\pi_\psi$ of $\SU$.

To begin,
define an antilinear operator $T_\psi$ on $V=\SU$ by
\begin{equation*}
 T_\psi(x)=x^*,\qquad x\in\SU.
\end{equation*}
Obviously, $T_\psi^2=1$.
Recall that the inner product on $\SU$ is given by
$\<x,y>=\psi(y^*x)$ and that the Haar state $\psi$ has the
property (see, e.g., \cite{KlimykS})
\begin{equation*}
\psi(xy)=\psi((k^{-2}\lt y \rt k^{-2})x),\quad x,y\in\SU.
\end{equation*}
Using this relation together with  Equations \eqref{eq:mod-alg},
\eqref{rmodul} and
\eqref{rho},  we compute
\begin{equation*}
\<y ,T_\psi x>=\psi(xy)=\psi((k^{2}\lt y^* \rt k^{2})^*x)
=\<x,  \pi_\psi(k^{2})\rho_\psi(k^{-2}) T_\psi(y)>,\quad x,y\in\SU.
\end{equation*}
Hence $T_\psi^*$ acts on $\SU$ by
$\pi_\psi(k^{2})\rho_\psi(k^{-2})T_\psi$ and $T_\psi$ is closeable.
In the Tomita-Takesaki theory, the closure of $T_\psi$ is
referred to as Tomita operator. By a slight abuse of notation, we 
denote in the sequel a closeable operator and its closure by the same symbol. 

Let $T_\psi=J_\psi | T_\psi|$ be the polar decomposition of the
Tomita operator. The antiunitary operator $J_\psi$ is known as
modular conjugation. Since
$T_\psi^*\hs T_\psi \lceil\hsp{}_{\SU}=\pi_\psi(k^{2})\hs
\rho_\psi(k^{-2})$, we have
\begin{equation*}
J_\psi\,x = T_\psi\, \pi_\psi(k^{-1})\,\rho_\psi(k)\, x
= \pi_\psi(k)\,\rho_\psi(k^{-1})\,T_\psi\, x,\qquad x\in \SU.
\end{equation*}
Equation \eqref{eq:mod-alg} implies that
$T_\psi\hs \pi_\psi(h)=\pi_\psi(S(h)^*)\hs T_\psi$ for
$h\in\su$. Therefore
\begin{equation}                                      \label{Jpsih}
J_\psi \pi_\psi(h)J_\psi^{-1}=\pi_\psi(kS(h)^*k^{-1}),
\end{equation}
exactly as in Equation \eqref{Jequi}. Note that $J_\psi^2=1$.

Our next aim is to define an antilinear ``Tomita" operator  $T$ on
the tensor product $W=\SU\otimes V_\half$ satisfying
$T\hs \pi(h)=\pi(S(h)^*)\hs T$ for
$h\in\su$. To begin, we look for an antilinear  operator $T_\half$ on
$V_\half$ such that
$T_\half\hs \sg_\half(h)=\sg_\half(S(h)^*)\hs T_\half$.
A convenient choice is given by
\begin{equation}                                     \label{Thalf}
T_\half\ket{\half,\half}=\im\hs q^{1/2}\ket{\half,-\half},\qquad
T_\half\ket{\half,-\half}=-\im\hs q^{-1/2}\ket{\half,\half}.
\end{equation}
Then $J_\half:=\sg_\half(k)T_\half=T_\half\sg_\half(k^{-1})$ is an
antiunitary operator and $J_\half^2=-1$. 

Since the antipode is a coalgebra \emph{anti-}homomorphism, 
i.e., $\cop S(h) =S(h_{(2)})\otimes S(h_{(1)})$, we combine 
$T_\psi\otimes T_\half$ with the flip operator on tensor products
and set 
\begin{equation*}
 T_0:=\tau\circ(T_\psi\otimes T_\half)\,:\,
\SU\otimes V_\half\rightarrow V_\half\otimes\SU, 
\end{equation*}
where $\tau$ is defined by $\tau (x\otimes y)=y\otimes x$.
By construction, the antilinear operator $T_0$ satisfies
$$
T_0\,\pi(h)= (\sg_\half\otimes \pi_\psi)(S(h)^*)\,T_0, \qquad h\in\su.
$$
To obtain a mapping from $W$ into itself, we compose $T_0$ 
with an operator intertwining the tensor product representations
$\sg_\half\otimes \sg_l$ and $\sg_l\otimes \sg_\half$. Such an operator
is provided by the universal R-matrix of $\su$
(see, e.g.,~\cite{KlimykS}).
For a tensor product representation  with  $\sigma_\half$ as left tensor
factor,  it can be expressed by
\begin{equation}                                     \label{R}
R=(\sg_\half\otimes\pi_\psi)\hs
\big(qfe\otimes k+qef\otimes k^{-1} + (q-q^{-1})q^{1/2} f\otimes e\big).
\end{equation}
Let $\hat R:=\tau\circ R$. 
It follows from the properties of the R-matrix (or can be checked
directly) that
\begin{equation*}
\pi(h)\circ\hat R
=\hat R\circ (\sg_\half\otimes \pi_\psi)(h), \qquad h\in\su.
\end{equation*}
Therefore the antilinear operator
\begin{equation*}
T:= \hat R\circ T_0
\end{equation*}
fulfills $T \pi(h)=\pi(S(h)^*)T$ for $h\in\su$ as required.

To describe the action of $T$ on $W$, we need at first explicit
formulas for the action of $T_\psi$ on $\SU$.
On writing the matrix element $t^l_{mn}$ in
\eqref{eq:matelt-onb} in terms of the generators of $\SU$
(see, e.g., \cite{DabrowskiLSSV} or \cite{KlimykS}), one easily sees that 
\begin{equation}                              \label{Tpsi}
T_\psi \,\ket{lmn} = (-1)^{2l+m+n} q^{m+n} \,\ket{l,-m,-n}.
\end{equation}
From \eqref{eq:spinor-basis-coeffs}--\eqref{eq:spinor-basis-up},
\eqref{R} and \eqref{Tpsi}, we obtain after a direct calculation
\begin{align}                                             \label{Tdn}
\begin{split}
T\, \ket{l\hs m\hs \nu\dn} &= \im^{2(2l-m-\nu)}\,
 q^{l+m+\nu+\half} \,\ket{l,-m,-\nu,\dn},\\
T\, \ket{l\hs m\hs \mu\up} &=  \im^{2(2l+m+\mu)}\,
 q^{-l+m+\mu-\half}  \,\ket{l,-m,-\mu,\up},    
 \end{split}       
\end{align}
where we used also the fact that $2(m+\nu)-1$ is an even integer.
Equation \eqref{Tdn} implies 
that $T$ maps $W_l^\up$ and $W_l^\dn$ into them\-selves.
As a consequence, $TD=DT$.

Remarkably, we even have $[\pi(x), T\pi(y)T^{-1}]=0$ for all
$x,y\in\SU$. To see this, one uses
$(\cop\otimes \id) \hat R=(\sum_i \id\otimes h_i\otimes g_i)
(\sum_j h_j\otimes \id \otimes g_j)$, where
$\hat R=\sum_i  h_i\otimes g_i$, which can be deduced from
general properties of R-matrices \cite{KlimykS}.
Then a straightforward computation shows that
$T\pi(y)T^{-1}(w\otimes v)=\sum_i w(h_i\lt y^*)\otimes \sg_\half
(g_i)v$ for all $w\otimes v\in\SU\otimes V_\half$. 
Since $T\pi(y)T^{-1}$ acts by
right multiplication on the first tensor factor, 
and $\pi(x)$ by left multiplication, it is clear that
$\pi(x)$ and $T\pi(y)T^{-1}$ commute.

Observe that
\begin{align}                                             \label{T*dn}
\begin{split}
T^*\, \ket{l\hs m\hs \nu\dn} &= \im^{2(2l+m+\nu)}\,
 q^{l-m-\nu+\half} \,\ket{l,-m,-\nu,\dn},\\
T^*\, \ket{l\hs m\hs \mu\up} &=  \im^{2(2l-m-\mu)}\,
 q^{-l-m-\mu-\half}  \,\ket{l,-m,-\mu,\up}.                
 \end{split}
\end{align}
In particular,  $T^*$ is densely defined, and $T$ is closeable.
By the convention made above, its closure will again be denoted by $T$. 
In analogy with the
Tomita operator $T_\psi$, define an antilinear operator $J$ by the
unique polar decomposition $T=J |T|$.
Comparing Equations \eqref{Tdn} and \eqref{T*dn} with
Equation \eqref{eq:J-formula} shows that this $J$ actually
coincides  with that from Section \ref{sec:real-struct-SU}. Moreover, $| T|$
is  given on $W$ by
\begin{equation}                                            \label{|T|}
| T|\hs w=\pi(k)\hs \rho(k^{-1})\hs q^{-D}\hs w,\qquad w\in W,
\end{equation}
where $\rho(h):=\rho_\psi(h)\otimes \id$ for $h\in\su$.
Thus we arrive at the following Proposition.
\begin{prop}                                              \label{JT}
The antilinear operator $J$ from Equation \eqref{eq:J-formula} can
be expressed by
\begin{equation*}
  J\hs w=T\hs \pi(k^{-1})\hs \rho(k)\hs q^{D}\hs w
  =\pi(k)\hs \rho(k^{-1})\hs q^{D}\hs T\hs  w,
\qquad w\in W.
\end{equation*}
\end{prop}

Proposition \ref{JT} yields another proof of the invariance
relation \eqref{Jequi}.
Since  $q^D$ and $\rho(k)$ commute with $\pi(h)$ and since
$T\hs \pi(h)=\pi(S(h)^*)\hs T$ for all $h\in\su$,
$J$ and $\pi(h)$ satisfy the same commutation relation as $J_\psi$
and $\pi_\psi(h)$ in Equation \eqref{Jpsih}.

We showed above that $[\pi(x), T\pi(y)T^{-1}]=0$ for all
$x,y\in\SU$. The 
operator
$J_0:=T\hs \pi(k^{-1})\hs \rho(k)$ still satisfies
$[\pi(x), J_0\pi(y)J_0^{-1}]=0$
for all $x,y\in\SU$ since
$\pi(k^{-1})\rho(k)\pi(y)\rho(k^{-1})\pi(k)
=\pi(k^{-1}\lt y\rt k^{-1})$ by the equivariance of
$\pi=\pi_\psi\otimes\id$. However, it was argued
in \cite{DabrowskiLSSV} that $J$ does not have this property.
This is due to the operator $q^D=|\hat R^*|^{-1}$ ensuring the
(anti)unitarity of $J$. To verify $|\hat R^*|=q^{-D}$,  observe that
$$
(T_\psi\otimes T_\half)^*w=-
(\pi_\psi(k^{2})\rho_\psi(k^{-2})T_\psi\otimes\sg_\half(k^2)T_\half )w
=-\pi(k^2)\rho(k^{-2})(T_\psi\otimes T_\half)w\quad w\in W.
$$
Moreover, $\hat R(T_\psi\otimes T_\half)=(T_\psi\otimes T_\half)\hat R^*$
and $\hat R^*(T_\psi\otimes T_\half)=(T_\psi\otimes T_\half)\hat R$ since
$S\otimes S\hs (\hat R)=\hat R$. Hence
$$
T^*\hs T\hs w=(T_\psi\otimes T_\half)^*\hs (T_\psi\otimes T_\half)\hs\hat
R\hs \hat R^*\hs w=
\pi(k^2)\rho(k^{-2})|\hat R^*|^2w, \quad w\in W.
$$
Comparing this equation with \eqref{|T|} gives
$|\hat R^*|=q^{-D}$ since $\pi(k)$ and $\rho(k)$ are invertible
on $W$.

%
\subsection{Equivariant real even spectral triple for $\pinf$}
                                              \label{sec:real-struct-podl}

For an $2^+$-summable even spectral triple with grading operator
$\gamma$, the requirements on a real structure $J$ include the commutation
relation $J\gamma=-\gamma J$.  By Proposition \ref{gradprop}, only
$(\pinf,\tH,\tilde D)$ admits a grading operator of an even spectral
triple. For this reason, we restrict the following discussion to the
equatorial Podle\'s sphere $\pinf$ although most of the results remain
valid in the general case.

We proceed as in Section \ref{sec:real-struct-Tomita} and define an
antilinear operator $\tilde T_{\psi}$ on $M_0=\pinf$ by
\begin{equation*}
\tilde T_{\psi}(x)=x^*,\qquad x\in\pinf.
\end{equation*}
By Equation \eqref{eq:mod-alg}, since $\pi_0(h)x=h\lt x$, we have
$\tilde T_\psi\hs \pi_0(h)=
\pi_0(S(h)^*)\hs \tilde T_\psi$
for all $h\in\su$.
From \cite[Lemma 6.3]{SchmuedgenWCPAPod}, it follows that
the Haar state $\tilde \psi$ on $\pinf$ satisfies
property
\begin{equation*}                                 
\tilde \psi(xy)=\tilde \psi((k^{-2}\lt y )x),\qquad x,y\in\pinf.
\end{equation*}
Analogously to Section  \ref{sec:real-struct-Tomita},
$\tilde T_{\psi}^*\lceil\hsp{}_{M_0}=\pi_0(k^2)\hs \tilde T_\psi$
and $\tilde T_{\psi}$ is
closeable (with closure denoted again by $\tilde T_{\psi}$). 
Moreover, $|\tilde T_{\psi}|\lceil\hsp{}_{M_0}=\pi_0 (k)$, and 
the anti\-unitary operator $\tilde J_{\psi}$
from the polar decomposition
$\tilde T_{\psi}
=\tilde J_{\psi}\hs |\tilde T_{\psi}|$
is given on $M_0$ by
\begin{equation*}
\tilde J_\psi\,x = \tilde T_\psi\, \pi_0 (k^{-1})\, x
= \pi_0 (k)\,\tilde T_\psi\hs x, \qquad x\in M_0. 
\end{equation*}

Since the entries of the R-matrix in \eqref{R} are elements from
$\su$, the restriction of $\hat R$ (again denoted by $\hat R$) to $\tilde
W=M_0\otimes V_\half$ leaves $\tilde W$ invariant.
Thus, with $T_\half$ from the previous subsection,
\begin{equation}                                    \label{tT}
\tilde T:= \hat R \,(\tilde T_\psi \otimes T_\half).
\end{equation}
defines an antilinear operator on $\tilde W$.
By construction,
$\tT\hs \tpi(h)=\tpi(S(h)^*)\hs \tT$ for all $h$ in $\su$.
Its adjoint $\tT^*$ acts on
$\tilde W$ by
\begin{equation*}
\tT^*w=-\tpi(k^2)\hs (\tilde T_\psi \otimes T_\half)\hs \hat R^*\hs w,
\qquad w\in \tW.
\end{equation*}
Recall that $|\hat R^*|=q^{-D}$,\,
$\hat R(\tT_\psi\otimes T_\half)=(\tT_\psi\otimes T_\half)\hat R^*$ and
$\hat R^*(\tT_\psi\otimes T_\half)=(\tT_\psi\otimes T_\half)\hat R$.
Hence
\begin{equation*}                                      
\tT^*\hs\tT\hs w = \tpi(k^2)\hs \hat R\hs\hat R^*\hs w =
\tpi(k^2)\hs q^{-2 D}\hs w, \quad w\in \tW.
\end{equation*}
Clearly, $\tT^*$ is densely defined and, therefore,
$\tT$ is closeable. Denoting its closure again by $\tT$,
we can write
$|\tT|\lceil\hsp{}_{\tW}\hs = \tpi(k)\hs q^{- \tilde D}$
since $D\lceil\hsp{}_{\tH}\hs =\tilde D$.

Now we define an antiunitary operator $\tJ$ by
the polar decomposition $\tT=\tJ\hs |\tT|$.
From the preceding, it follows that
\begin{equation}                                         \label{tJw}
\tJ\hs w=\tT\hs \tpi(k^{-1})\hs q^{\tilde D}\hs w =\hat R \,
(\tilde T_\psi \otimes T_\half)\hs \tpi(k^{-1})\hs q^{\tilde D}\hs w,
\qquad w\in\tW.
\end{equation}

Our next aim is to give explicit formulas for the action of $\tJ$. Let
$\ket{lm\dn}$ and $\ket{lm\up}$ denote the vectors  defined by Equations
\eqref{lldn} and \eqref{lmupdn}. The set of all these
vectors forms an orthonormal basis for~$\tH$. Inserting \eqref{lldn} 
into \eqref{lmupdn}, one easily verifies that
\begin{align}                                         \label{lmd}
\ket{l\hs m\hs \dn}
&:= C_{lm } \, v^{l-\half}_{m +\half,0} \ox \ket{\half,-\half}
+ S_{lm } \,v^{l-\half}_{m -\half,0} \ox \ket{\half,+\half}, \\
\ket{l\hs m\hs \up}                                    \label{lmu}
&:= - S_{l+1,m } \,v^{l+\half}_{m +\half,0} \ox \ket{\half,-\half}
+ C_{l+1,m } \,v^{j+\half}_{m -\half,0} \ox \ket{\half,+\half}
\end{align}
with $C_{lm }$ and $S_{lm }$ given by \eqref{eq:spinor-basis-coeffs}.

To determine $\tT_\psi$, we use the identification $M_0\hsp =\hsp \pinf$.
Then
$v^0_{0,0} \hsp =\hsp 1$ and, thus,
$$
v^{l+1}_{l+1,0}\hsp =\hsp
(\Pi_{k=0}^l\alpha_0(k))^{-1}\pi_0(x_1)^lv^0_{0,0}\hsp =\hsp
(\Pi_{k=0}^l\alpha_0(k))^{-1}x_1^l.
$$ 
Similarly,
$v^{l+1}_{-l-1,0}\hsp =\hsp (\Pi_{k=0}^l\alpha_0(k))^{-1}x_{-1}^l$.
This gives $\tT_\psi v^{l}_{l,0}\hsp =\hsp (-q)^l v^{l}_{-l,0}$ since
$x_1\hsp =\hsp -qx_{-1}^*$. Computing both sides of
$\tT_\psi\hs \tpi(e)^k\hs v^{l}_{l,0}\hsp =\hsp
(-q)^{-k}\hs \tpi(f)^k\hs\tT_\psi\hs v^{l}_{l,0}
\hsp =\hsp (-q)^{l-k}\hs\tpi(f)^k\hs v^{l}_{-l,0}$, we finally get
$$
\tT_\psi\hs v^{l}_{m,0}\hs =\hs (-q)^m \hs v^{l}_{-m,0}, \qquad
l\in \N_0,\ \, m =-l,\ldots,l.
$$

Using these formulas, the action of
$\tT=\hat R \,(\tilde T_\psi \otimes T_\half)$
on $\ket{lm\dn}$ and $\ket{lm\up}$ can be computed directly.
Analogously to Equation \eqref{Tdn}, we find
\begin{equation*}                                          
\tT\, \ket{l\hs m\dn} = \im^{2m}\,
 q^{l+m+\half} \,\ket{l,-m,\dn},\quad
\tT\, \ket{l\hs m\up} =  -\im^{2m}\,
 q^{-l+m-\half}  \,\ket{l,-m,\up}.
\end{equation*}
Consequently, by \eqref{tJw},
\begin{equation*}                                          
\tJ\, \ket{l\hs m\dn} = \im^{2m}\,  \ket{l,-m,\dn},\qquad
\tJ\, \ket{l\hs m\up} =  -\im^{2m}\, \ket{l,-m,\up}.
\end{equation*}
Therefore, by \eqref{lmv} (up to unitary equivalence), 
\begin{equation*}
                   \tJ\,  v^l_{m,\pm\half} =\im^{2m}\, v^l_{-m,\mp\half}, 
\end{equation*}
where $l=\half,\sesq,\ldots$ and $m=-l,\dots,l$.

The last equation shows that  $\tJ$ coincides with
the real structure defined in \cite{DabrowskiLPS}. The results
in \cite{DabrowskiLPS} (or \cite{DDLW})
tell us that $[\tpi(a),\tJ\tpi(b)\tJ^{-1}]$ and
$[[D,\tpi(a)],\tJ\tpi(b)\tJ^{-1}]$ are infinitesimals of arbitrary high
order for all $a,b\in\pinf$.

Finally let us discuss how $\tT$ and $\tJ$ are related to $T$ and $J$ 
from Section \ref{sec:real-struct-Tomita}. 
Since $\tT_\psi=T_\psi\lceil_{\pinf}$, it follows from the definitions 
that $T\lceil\hsp{}_\tW=\tT$. 
In particular, as shown above, 
$[\pi(x), \tT\pi(y)\tT^{-1}]=0$ for all
$x,y\in\pinf$. On the other hand, we do not have
$J\lceil\hsp{}_\tW=\tJ$. 
This is due to the fact that
the adjoint $T_\psi^*(x)=k^{2}\lt x^* \rt k^{2}$ does not map $\pinf$ into itself.  
In general, $y \rt k^{2}\notin \pinf$ for $y\in\pinf$.

Summarizing our conclusions, we can now state the main theorem of this
paper.

\begin{thm} Let  $(\SU,\H,D)$ denote
the spectral triple described in Section \ref{sec:Dirac-oper-SU}.
The embedding $\podl\otimes V_\half\subset\SU\otimes V_\half$
gives rise to  an equivariant real even spectral triple
$(\podl,\tH,\tD,\tJ,\gamma)$ if and only if $c=\infty$.
The equivariant representation $\tpi$ on $\pinf\otimes V_\half$ is given
by restricting the *-representation of $\SU\rtimes \su$ on
$\SU\otimes V_\half$  to a *-rep\-re\-sen\-ta\-tion of
$\pinf\rtimes \su$ on
$\pinf\otimes V_\half$. The Dirac operator $\tD$ is the closure of the
restriction of $D$ to the invariant subspace $\pinf\otimes V_\half$.
The decomposition of $\pinf\otimes V_\half$ into eigenspaces
corresponding to the eigenvalues $\pm1$ of $\gamma$  coincides
with the decomposition into subspaces corresponding to
irreducible *-representations  of $\pinf\rtimes \su$. The real structure
$\tJ$ is the antiunitary part of the equivariant (closed) Tomita operator
defined in Equation \eqref{tT}. The commutators
$[\tpi(a),\tJ\tpi(b)\tJ^{-1}]$ and
$[[\tD,\tpi(a)],\tJ\tpi(b)\tJ^{-1}]$ are infinitesimals of arbitrary high
order for all $a,b\in\pinf$.
\end{thm}

\section*{Acknowledgments} 
The author gratefully acknowledges the useful comments of an anonymous referee. 
This work was carried out with partial financial support from the DFG (fellowship WA 1698/2-1), 
from the CIC of the Michoacan University (project ''Enlaces entre geometr\'ia no-conmutativa 
y ecuaciones de f\'isica matem\'atica``), and from CORDIS (FP7, PIRSES-GA-2008-230836).

\end{document}